\documentclass[a4paper,12pt]{amsart}

\usepackage{geometry}
\usepackage{mathtools}
\usepackage{appendix}
\usepackage{tikz}
\usetikzlibrary{arrows}
\usetikzlibrary{patterns}
\usetikzlibrary{hobby}
\usetikzlibrary{decorations.pathreplacing}
\usetikzlibrary{arrows.meta}
\mathtoolsset{showonlyrefs}

\usepackage{graphicx}
\usepackage{esint}
\usepackage{amsmath,amssymb,amsthm}
\usepackage{enumerate}
\usepackage{hyperref}
\usepackage[active]{srcltx}
\usepackage[T1]{fontenc}
\usepackage{color}
\usepackage[makeroom]{cancel}

\newtheorem{theo}{Theorem}

\newtheorem{lem}{Lemma}[section]

\newcommand{\eps}{\varepsilon}
\newcommand{\ed}{{\eps,\delta}}

\newcommand{\qqtext}[1]{\qquad\mbox{#1}\qquad}

\newcommand{\rd}{\mathrm d}
\newcommand{\R}{\mathbb{R}}

\newcommand{\narrowcv}{\overset{\ast}{\rightharpoonup}}

\newcommand{\M}{\mathcal{M}}
\newcommand{\A}{\mathcal{A}}

\newcommand{\E}{{\mathcal{E}}}
\newcommand{\I}{{\mathcal{I}}}

\renewcommand{\H}{{\mathcal{H}}}
\DeclareMathOperator{\argmin}{Argmin}

\newcommand{\vol}{V}

\renewcommand{\P}{\mathcal{P}}

\newcommand{\g}{\gamma}
\renewcommand{\O}{M}

\newcommand{\pM}{{\partial M}}

\DeclareMathOperator{\Ric}{Ric}

\def\dive{\operatorname{div}}

\numberwithin{equation}{section}
\begin{document}
%

\title[Entropy and Fisher information in non-convex domains]{Entropy and Fisher information in non-convex domains: one chain to rule them all}
\date{}

\author[J.B. Casteras]{Jean-Baptiste Casteras}
\address{CEMS.UL, Faculdade de Ci\^encias da Universidade de Lisboa, Edificio C6, Piso 1, Campo Grande 1749-016 Lisboa, Portugal
}
\email{jeanbaptiste.casteras@gmail.com}

\author[M. Flaim]{Marco Flaim}
\address{Institut f\"ur Angewandte Mathematik, Universit\"at Bonn, Bonn, Germany}
\email{flaim@iam.uni-bonn.de}

\author[L. Monsaingeon]{L\'{e}onard Monsaingeon}
\address{Grupo de F\'isica Matem\'atica, Departamento de Matem\'atica, Instituto Superior T\'ecnico, Av. Rovisco Pais 1049-001 Lisboa, Portugal
and
Institut \'Elie Cartan de Lorraine, Universit\'e de Lorraine, Site de Nancy B.P. 70239, F-54506 Vandoeuvre-l\`es-Nancy Cedex, France
}
\email{leonard.monsaingeon@tecnico.ulisboa.pt}

\begin{abstract}
We prove that the (square root) Fisher information functional is a strong Wasserstein upper gradient of the entropy on non-convex Riemannian domains, completely avoiding $\lambda$-displacement convexity arguments.
Along the way we obtain a novel quantitative short-time control of the Fisher information along the Neumann heat flow, and establish an exact chain rule under stronger $AC_2$ assumptions typically satisfied by curves of measures obtained as limits of JKO schemes.

\end{abstract}

\maketitle

\section{Introduction}
Since the groundbreaking work of R. Jordan, D. Kinderlehrer, and F. Otto \cite{JKO} it is known that many Fokker-Planck equations can be considered as variational gradient flows of entropy functionals with respect to Wasserstein metrics, see e.g. \cite{carrillo2003kinetic,matthes2009family,cances2017incompressible,casteras2023hidden,casteras2025sticky} to name but just a few, the classical textbooks \cite{OTAM,AmbBruSem21,figalli2021invitation}, and the excellent survey \cite{santambrogio2017euclidean}.

Putting a formally identified Wasserstein gradient flow structure on solid grounds and rigorously connecting the variational structure with PDEs and dissipation properties usually goes through the so-called \emph{JKO scheme}: given a small time-step $\tau>0$ one solves the recursive implicit Euler scheme
$$
\mu^{n+1}=\underset{\mu\in \P(X)}{\argmin}\left\{\frac{W^2(\mu,\mu^n)}{2\tau}+\E(\mu)\right\}
$$
and tries recovering in the small time-step limit a solution $\mu_t=\lim\limits_{\tau\to 0}\mu^\tau_t$, $t\in[0,T]$, of the evolution equation.
Here $X$ is an underlying space over which the Wasserstein distance $W$ is defined, $\mu^n\in \P(X)$ is the previous time-step, and $\E:\P(X)\to \R\cup\{+\infty\}$ is the driving functional often referred to as \emph{entropy}.
This goes back to De Giorgi \cite{de1980problems,de1993new} in abstract metric spaces under the name of \emph{minimizing movement}, and the theory was further developed in \cite{AGS}.
At this stage there are usually two different directions in which the analysis can go:
one can either
1) try to extract enough information from the Euler-Lagrange equations by whatever means available (typically leveraging very geometric features of optimal transport theory), and pass to the limit $\tau\to 0$ directly in the optimality conditions to retrieve the sought evolution PDE,
or
2) resort to more theoretical approaches and leverage the abstract theory of gradient flows in metric spaces \cite{AGS}.
In any case one typically ends-up with suitable dissipation estimates, and is able to show that the limiting curve $\mu_t=\lim\limits_{\tau\to 0}\mu^\tau_t$ satisfies an \emph{Energy Dissipation Inequality} (EDI).
This typically takes the form
\begin{equation}
\label{eq:EDI}
\frac{d}{dt}\E(\mu_t)\leq -\frac 12 |\dot\mu_t|^2 -\frac 12 |\partial\E|^2(\mu_t),
\tag{EDI}
\end{equation}
where $|\dot\mu_t|$ and $|\partial\E(\mu_t)|$ denote the metric speed and (descending) slope of $\E$.
One can also replace $|\partial \E|$ by a weaker notion of \emph{upper gradient}, leading to the notion of \emph{curve of maximal slope} \cite{AGS}.
We should point out that this is not merely a theoretical tool, and dissipation properties encoded in abstract upper-gradients or metric speeds can sometimes give access to very practical features of the model under investigation, for example boundary conditions as in \cite{casteras2025sticky,erbar2024gradient}.

In order to connect this rather abstruse notion of EDI/maximal slope solutions with more hands-on PDE considerations, one usually proceeds as follows.
From the JKO scheme one automatically gets that the limit curve $\mu=\lim\limits_{\tau\to 0}\mu^\tau\in AC_2([0,T];W)$ is absolutely continuous in time with respect to the Wasserstein distance.
Via Benamou-Brenier-type results this gives a continuity equation $\partial_t\mu_t+\dive(\mu_t v_t)=0$ for some velocity field $v=v_t(x)$ optimally representing $|\dot\mu_t|=\|v_t\|_{L^2(\mu_t)}$ \cite{BB,Lisini2007}.
Assuming that one could moreover identify the squared slope $|\partial\E|^2(\mu)$ with the Fisher information functional
$$
\I(\mu)=\int_X \left|\nabla\frac{\delta\E}{\delta\mu}(\mu)\right|^2\,\rd\mu
$$
and establish a suitable chain-rule, allowing to differentiate $\frac{d}{dt}\E(\nu_t)=\int_X \nabla\frac{\delta\E}{\delta\mu}(\nu_t)\cdot u_t\,\rd \nu_t$ along \emph{arbitrary} curves $\partial_t\nu_t+\dive(\nu_t u_t)=0$, \eqref{eq:EDI} forces equality in the $L^2(\mu_t)$ Cauchy-Schwarz inequality.
This singles out the driving velocity field as $v_t=-\nabla\frac{\delta\E}{\delta\mu}(\mu_t)$, which is exactly the sought PDE $\partial_t\mu_t=-\dive(\mu_tv_t)=\dive(\mu_t\nabla\frac{\delta\E}{\delta\mu}(\mu_t))$ that one tries to identify as a Wasserstein gradient-flow to begin with.

However, one very delicate point is precisely to justify such a chain rule, whether it be an exact chain rule $\frac{d}{dt}\E(\nu_t)=[\dots]$ or an upper chain rule $\left|\frac{d}{dt}\E(\nu_t)\right|\leq [\dots]$.
In Euclidean spaces $X=\R^d$ (or in general Hilbert spaces), it is well-known that for the Boltzmann entropy $\H(\mu)=\int \rho\log\rho\,\rd x$ the slope $|\partial \H|^2(\mu)$ is given by the Fisher information functional $\I(\mu)=\int|\nabla\log\rho|^2\rd\mu$, where $\rho$ is the density of $\mu=\rho(x)\rd x$ w.r.t. the Lebesgue measure (see e.g. \cite[Thm 15.25]{AmbBruSem21} and \cite[Thm 10.4.9]{AGS} for Gaussian measures, also \cite[Prop. 10.3.18]{AGS} for an exact chain rule).
This can be extended to relative entropies in complete manifolds under Bakry-\'Emery conditions, more precisely under lower Ricci bounds \cite{sturmvonrenesse}.
This is very much related to R. McCann's notion of \emph{displacement convexity} \cite{mccann1997convexity}, which in the context of abstract metric spaces is rather enforced as \emph{$\lambda$-convexity} (or \emph{convexity along generalized geodesics}), see again \cite{AGS}.
However, even in smooth Euclidean domains $\Omega\subset\R^d$, all these notions of convexity completely fail unless $\Omega$ is convex.
More precisely, the entropy $\int_\Omega\rho\log\rho \,\rd x$ fails to be $\lambda$-convex for any $\lambda\in \R$, and this corresponds at least formally to $\mathsf{RCD}(-\infty,N)$ spaces.
These are too singular to fall within the theory \cite{AGS1,AGS3}.

Nevertheless, it was shown in \cite{LS18} that, even in non-convex Riemannian domains, it holds $|\partial\H|^2(\mu)=\I(\mu)$, and that an $\mathrm{EDI}$ solution for $\H$ in the Wasserstein space is an $\mathrm{EDI}$ solution for the Cheeger energy in $L^2$. This is achieved by approximating the distance with a sequence of conformal distances for which the domain is convex, and by applying the convex theory along this sequence. In this work, we tackle this problem by taking a different route. From the $\mathrm{JKO}$ scheme, one can obtain, in the limit, a solution to the $\mathrm{EDI}$, provided that the Fisher information is used in place of the descending slope. This can be obtained thanks to the inequality $\I(\mu) \leq |\partial\H|^2(\mu)$, \cite[Thm.~7.4]{AGS1}, and to the lower semicontinuity of $\I$. We then show that $\I$ is an upper gradient of $\H$, by working directly on the non-convex domain and avoiding any displacement convexity argument. This approach leads also to an exact chain-rule for more regular curves. These results allow one to pass from the $\mathrm{EDI}$ formulation to the $\mathrm{EDE}$, and even to the solution of the $\mathrm{PDE}$. Finally, we apply this method to more general entropies, in particular to the Renyi entropy, and to the case of an additional external energy term.

Another topic very tightly connected to displacement convexity is that of the growth/decay of the Fisher information along the heat flow.
At least in convex domains $\Omega\subset\R^d$, it is folklore knowledge in the PDE community that the Fisher information is non-increasing along the heat-flow (see e.g. \cite[Cor.~2.6]{santambrogio2024strong}, taking $\Omega=\{h<0\}$ with $h$ convex).
This is also known to the probabilistic community under Bakry-\'Emery lower bounds in complete, boundaryless manifolds, $\Ric(x)+\operatorname{Hess} V(x)\geq -K\implies \I(\mu_t)\leq  e^{2Kt}\I(\mu)$ (e.g. in whole Euclidean spaces with Gaussian reference measures), corresponding to $-K$-displacement convexity of the relative entropy.
In case of manifolds with boundaries $\partial M\neq \emptyset$ this seems to be limited to folklore, barely appears explicitly stated in the literature, and only works in convex domains.
For the sake of completeness we opted for providing a general statement in convex Riemannian domains, Theorem~\ref{theo:Fisher_decay_convex}, and included an elementary probabilistic proof.
However, and again to the best of our knowledge, no such estimate is available in the non-convex case.
We will establish a novel $\exp(\mathcal O(\sqrt t))$ growth estimate in non-convex domains, which shows that the geometry of the boundary actually determines the growth rate at leading order in the short-time regime.
We believe that this new estimate is of independent interest, although we only use it here for purely technical reasons in the proof of our chain rule (quantification of the heat-flow regularization effect).
\\

Our main results in non-convex Riemannian domains can be summarized as:
\begin{enumerate}
    \item 
    a quantified growth estimate of the Fisher information along the heat-flow, Theorem~\ref{theo:Fisher_decay}
    \item
    an upper chain rule for the entropy, where we show that $\mathfrak g(\mu)=\sqrt{\I}(\mu)$ is a Wasserstein upper-gradient for the entropy $\H$, Theorem~\ref{theo:fisher_upper_grad}
    \item
    an exact chain rule $\frac{d}{dt}\H(\mu_t)=[\dots]$ along $AC_2$ curves $(\mu_t)_{t\in [0,T]}$ under a natural time-integrability condition for the Fisher information, Theorem~\ref{theo:chain_rule_G}
\end{enumerate}
\bigskip
The rest of the paper is organized as follows:
in Section~\ref{sec:notation_results} we introduce our notations, we fix once and for all the setup to be used throughout the whole paper, and we state carefully our main results.
Section~\ref{sec:proofs} contains the proofs, that are extended to more general entropies in Section \ref{sec:general_ent}.
Appendix~\ref{sec:appendix} contains a few auxiliary technical lemmata.

\section{Notations and main results}
\label{sec:notation_results}
Let us fix here our notations and assumptions, enforced throughout the whole paper unless explicitly stated otherwise
\begin{itemize}
    \item 
$(M,g)$ is a smooth, compact manifold, possibly with boundary $\partial M\neq\emptyset$.
For notational convenience we abuse notations and indistinctly write $g_x(u,v)=\langle u,v\rangle =u \cdot v$ for the scalar product between two tangent vectors $u,v\in T_xM$.
Similarly, we omit the basepoint-dependence and simply write $|u|^2=g_x(u,u)$ for the squared tangent norm of $u\in T_xM$.
\item
The Riemannian gradient and Laplace-Beltrami operator are simply denoted $\nabla f,\Delta f$ for sufficiently smooth functions $f$.
\item
We always denote by $\Ric,\operatorname{I\!I}$ the Ricci curvature and second fundamental form on $M,\partial M$, respectively.
Here and throughout we assume Ricci and convexity lower bounds
\begin{equation}
    \label{eq:assumptions_ricci_curvature}
    \Ric(x)\geq -K\text{ on }M
    \qqtext{and}
    \operatorname{I\!I}(x)\geq -S \text{ on }\partial M
\end{equation}
for some finite $K,S\in \R$.
A convex domain simply means $S=0$, and unless explicitly stated otherwise we always consider non-convex domains $S>0$.
\item
The Riemannian volume form is denoted $V_g$, or simply $V$.
A generic probability measure over $M$ is denoted $\mu\in \P(M)$.
In case $\mu\ll V$ is absolutely continuous we write $\rho=\frac{\rd \mu}{\rd V}$ for the Radon-Nikodym density, i.e. $\rd \mu(x)=\rho(x) \rd V(x)$ or simply $\mu=\rho V$.
\item
We write $\P^\infty(M)$ for the set of probability measures of the form $\mu=\rho V$ with $\rho\in C^\infty(M)$ bounded from above and from below.
\item
The intrinsic Riemannian distance is
$$
d^2(x,y)=\min\left\{\int_0^1 g(\dot\g_t,\dot\g_t)\rd t\text{ s.t. }\g_0=x,\g_1=y\right\}
$$
\item
We denote the corresponding Wasserstein distance by
$$
W^2(\mu,\nu)=\min\left\{\iint d^2(x,y)\,\rd\pi(x,y)\text{ s.t. }\pi\in\P(M\times M),\,\pi_x=\mu,\pi_y=\nu\right\}
$$
and refer e.g. to \cite{OTAM,AGS,Vil09} for a detailed account on the theory of optimal transportation.
\item
We say that a curve $\mu=(\mu_t)_{t\in [0,T]}$ is absolutely continuous (relatively to the Wasserstein distance) and write $\mu\in AC([0,T];W)$ if there exists a function $m\in L^1(0,T)$ s.t.
\begin{equation}
\label{eq:def_AC}
W(\mu_{s},\mu_{t})
\leq
\int_s^t m(\tau)\,\rd \tau
\end{equation}
for all $0<s\leq t<T$.
In that case the \emph{metric speed}
$$
|\dot\mu_t|=\lim\limits_{h\to 0}\frac{W(\mu_{t},\mu_{t+h})}{h}
$$
exists a.e. $t\in (0,T)$ and is the smallest function satisfying \eqref{eq:def_AC}, see e.g. \cite[thm. 1.1.2]{AGS}.
We say that $\mu\in AC_p$, $p>1$ if moreover $|\dot\mu|\in L^p(0,T)$.
\item
The entropy $\H(\mu)$ of $\mu\in \P(M)$ is
$$
\H(\mu)=
\begin{cases}
    \int_M \rho \log\rho \,\rd V & \text{if }\mu=\rho V\\
    +\infty &\text{else,}
\end{cases}
$$
sometimes also denoted $\H(\mu\vert V)$ (the relative entropy of $\mu$ w.r.t the reference volume measure $V$).
\item
We say that $\mu=\rho V$ has logarithmic derivative $u$ if the distributional gradient $\nabla\mu\in L^1_{loc}(M)$ is absolutely continuous w.r.t. $\mu$ with corresponding Radon-Nikodym derivative
\begin{equation}
\label{eq:log_derivative}
    u=\frac{\rd\nabla\mu}{\rd\mu}.
\end{equation}
In that case we write for convenience $u="\nabla\log\rho"$ with a clear abuse of notations.
Of course if $\rho(x)$ is smooth enough and positive then indeed $u=\nabla\log\rho$ coincides with the classical gradient of the well-defined function $\log\rho$.
\item
The Fisher information functional is
$$
\I(\mu)=\begin{cases}
\int_M |\nabla\log\rho|^2\rd\mu & \text{if }\mu\text{ has logarithmic derivative}
\\
+\infty &\text{else}.
\end{cases}
$$
Note that $\I(\mu)$ can be $+\infty$ if $\nabla\log\rho\not\in L^2(\mu)$.
Equivalently, one has the alternative representation
\begin{equation}
\label{eq:def_Fisher_H1_u}
\I(\mu)=
\begin{cases}
    4\int_M |\nabla u|^2\,\rd V & \text{if }\mu=\rho V\text{ with }u=\sqrt{\rho}\in H^1(M)\\
    +\infty &\text{else}
\end{cases},
\end{equation}
where $H^1(M)$ is the standard Sobolev space $H^1=W^{1,2}$.
\item
The Neumann heat-flow (with generator $\Delta$) of a function $f$ is denoted $P_t f$, while the dual flow acting on measures $\mu\in \P(M)$ is denoted $P^*_t\mu$.
For $\mu=\rho V$ the duality can be expressed as $P_t^*\mu=(P_t\rho) V$.
\item
We say that a $\P(M)$-valued curve $\mu=(\mu_t)_{t\in [0,T]}$ and a $TM$ vector-valued momentum measure $F=(F_t)_{t\in [0,T]}$ satisfy the continuity equation
\begin{equation}\label{eq:CE}
\partial_t\mu_t+\dive F=0
\end{equation}
with zero flux boundary condition on $\pM$ if
$$
\int_M \varphi \,\rd\mu_{t_1}-\int_M \varphi\, \rd\mu_{t_0}- \int_{t_0}^{t_1}\int_M \nabla\varphi\cdot \rd F\,\rd t=0
$$
for all $0\leq t_0\leq t_1\leq T$ and $\varphi\in C^1(M)$.
\item
For $a\in \R,b\in T_xM$ we denote
$$
A(a,b)
=\begin{cases}
\frac{|b|^2}{a} & \text{ if }a>0,\\
0 & \text{if }a=0,b=0,\\
+\infty & \text{else}.
\end{cases}
$$
Here we omit the $x$-dependence for simplicity, but typically one should keep in mind that we mean $|b|^2=g_x(b,b)$ whenever $b\in T_xM$.
Note that this is a convex, 1-homogeneous function.
Given a vector-valued momentum $F$ and $\mu\in \P(M)$ we write the kinetic action as
$$
\A(\mu,F)=\int_M\frac{|F|^2}{\mu}=\int_M A\left(\frac{\rd\mu}{\rd\lambda},\frac{\rd F}{\rd\lambda}\right) \rd\lambda
$$
for any positive measure $\lambda\in \M^+(M)$ such that $|F|,\mu\ll\lambda$.
By one homogeneity of $A$ this is independent of $\lambda$.
Whenever $\A(\mu,F)<+\infty$ it is easy to check that $F\ll\mu$, in which case one can simply take $\lambda=\mu$ and the kinetic action is best represented in terms of the velocity $v=\frac{\rd F}{\rd \mu}$ (rather than the momentum $F=\mu v$) as
$$
\A(\mu,F)=\int_M\frac{|F|^2}{\mu}=\int_M|v|^2\rd \mu,
\hspace{1cm}
F=\mu v.
$$
\end{itemize}
With this being said, our first result reads
\begin{theo}[Fisher information decay, convex domains]
\label{theo:Fisher_decay_convex}
Assume that $M$ is convex, i.e. $S=0$ with $\operatorname{I\!I}(x)\geq 0$, and let $\Ric(x)\geq -K$.
For any $\mu\in \P(M)$, the Neumann heat-flow $\mu_t=P^*_t\mu$ started from $\mu$ satisfies
$$
\I(\mu_t)\leq e^{2Kt}\I(\mu).
$$
\end{theo}
Although we are mainly interested in possibly non-convex domains ($S>0$) we could not find anywhere in the literature the explicit statement in this general form, so we decided to include this auxiliary result both for the sake of completeness and further bibliographical convenience.

Focusing now on non-convex domains, our next result is the exact analogue on non-convex domains and is completely new, even in a PDE framework.
\begin{theo}[Fisher information decay, non-convex domains]
\label{theo:Fisher_decay}
Assume that $M$ satisfies $\operatorname{I\!I}(x)\geq -S$ with $S>0$, and any Ricci lower bound.
There exists a small $t_0>0$ such that, for any $\mu\in \P(M)$, the Neumann heat-flow $\mu_t=P^*_t\mu$ started from $\mu$ satisfies
$$
\I(\mu_t)\leq e^{4S\sqrt{t/\pi}+\mathcal O(t)}\I(\mu),
\hspace{1cm} \forall\,t\in [0,t_0].
$$
Here the $\mathcal O(t)$ term is uniform in $\mu\in \P(M)$.
\end{theo}
We believe this an interesting estimate in its own right, but we will mainly use as a technical tool to establish our second main result. Note that the statement already follows from \cite{LS18}.
\begin{theo}[Fisher is a strong upper gradient]
\label{theo:fisher_upper_grad}
The square-root Fisher Information
$$
\mathfrak g(\mu)\coloneqq \sqrt{\I(\mu)}
$$
is a strong Wasserstein upper gradient for the entropy $\H$, i.e., if $\mu\in AC([0,T];W)$ then
\begin{equation}
\label{eq:chain_rule_upper_grad}
\left|\H(\mu_{t_1})-\H(\mu_{t_0})\right|
\leq \int_{t_0}^{t_1}\mathfrak g(\mu_t)|\dot\mu_t|\,\rd t,
\qquad
\forall\,0<t_0\leq t_1<T.
\end{equation}
\end{theo}
This is typically required to guarantee that, when using $\mathfrak g(\mu)=\sqrt{\I}(\mu)$ as dissipation measurement, EDI is a good notion of solution.
For example this shows that an EDI solution in time interval $[0,T]$ is in fact a stronger \emph{Energy Dissipation Equality} (EDE) solution (and not just EDI) in any subinterval $[t_0,t_1]\subseteq [0,T]$, see e.g. \cite[Appendix A]{cances2024discretizing} for a detailed exposition in the case of Fokker-Planck equations.
Moreover, since $\I(\mu)$ and $\A(\mu,F)$ are convex in $\mu$ and $(\mu,F)$, respectively, our upper chain rule \eqref{eq:chain_rule_upper_grad} can be used together with a trick by N. Gigli \cite{gigli2010heat} to show that EDI solutions are unique, see again \cite[Appendix A]{cances2024discretizing}.
\\

Our last main result is a more PDE-oriented and exact chain-rule:
\begin{theo}[exact chain rule under $AC_2$/Fisher conditions]
\label{theo:chain_rule_G}
Assume that $\mu=(\mu_t)_{t\in[0,T]}$ solves the continuity equation $\partial_t\mu+\dive G=0$ with finite kinetic energy
\begin{equation}
\label{eq:finite_energy_G}
   \int_0^T \int_M\frac{|G_t|^2}{\mu_t}\rd t<+\infty
\end{equation}
and Fisher information
\begin{equation}
\label{eq:finite_Fisher}
    \int_0^T \I(\mu_t)\rd t<+\infty.
\end{equation}
Let $F_t$ be the optimal representative of $(\mu_t)_{t\in[0,T]}$ in the sense of Lemma \ref{lem:BBnon-convex}.
Then $t\mapsto \H(\mu_t)$ is absolutely continuous and
\begin{equation}
\label{eq:chain_rule_G}
    \frac{d}{dt}\H(\mu_t)=\int_M \nabla \log\rho_t \cdot \rd F_t 
\end{equation}
for a.e. $t\in (0,T)$, the right-hand side belonging to $L^1(0,T)$.

If we additionally assume that for a.e. $t\in (0,T)$, the velocity field $u_t=\nabla\log\rho_t\in L^2(\mu_t)$ belongs to $\overline{\{\nabla\phi,\,\phi\in C^1(M)\}}^{L^2(\mu_t)}$, it holds
\begin{equation}
    \frac{d}{dt}\H(\mu_t)=\int_M \nabla \log\rho_t \cdot \rd G_t 
\end{equation}
for a.e. $t\in (0,T)$.
\end{theo}

Theorem \ref{theo:chain_rule_G} allows one to recover the weak $\mathrm{PDE}$ from the $\mathrm{EDI}$, obtained from the $\mathrm{JKO}$ scheme, even without assuming $\nabla\log\rho_t\in \overline{\{\nabla\phi,\,\phi\in C^1(M)\}}^{L^2(\mu_t)}$. Indeed, assume $(\mu_t)_{t\in[0,T]}$ solves the $\mathrm{EDI}$, in particular it is $AC_2$ and we can find its optimal representative $F$ in the continuity equation. After applying the chain rule, the $\mathrm{EDI}$ forces rigidity in Young's inequality: for $F_t=v_t\mu_t$ we have
    \begin{align}
    - \frac{1}{2}\int_{t_0}^{t_1}(|\dot\mu_t|^2 + \I(\mu_t))\rd t &= - \frac{1}{2}\int_{t_0}^{t_1}\int \left(\left|v_t\right|^2 + |\nabla\log(\rho_t)|^2\right)d\mu_t\rd t \\&\overset{\text{Young}}{\leq} \int_{t_0}^{t_1} \int_M \nabla \log\rho_t \cdot \rd F_t = \H(\mu_{t_1}) - \H(\mu_{t_0}) \\
    & \hspace{5cm}\overset{\mathrm{EDI}}{\leq} - \frac{1}{2}\int_{t_0}^{t_1} (|\dot\mu_t|^2 + \I(\mu_t))\rd t \,.
    \end{align}
    Therefore $v=-\nabla\log\rho$ (in the $L^2$ sense) and from the definitions \eqref{eq:log_derivative}, \eqref{eq:CE} we get that, for every test function $\varphi \in C^2(M)$,
    \begin{align}
        \int_M\varphi \rd\mu_{t_1} - \int_M\varphi \rd\mu_{t_0} = \int_{t_0}^{t_1} \int_M \Delta\varphi \rd\mu_t \rd t \,.
    \end{align}

This approach completely bypasses more advanced optimal transport arguments (roughly speaking, writing down the Euler-Lagrange equations at each JKO step and then proving that this optimality condition gives the desired PDE in the limit \cite{OTAM}), which often requires convexity of the domain for convenience.

Also, we believe that the assumption in the second statement that $\nabla\log\rho$ lies in the closure of gradients is redundant, and we conjecture that this is automatic as soon as $\I(\mu)<+\infty$.
However we did not manage to prove this in full generality so for the sake of rigor we added this extra assumption, which might be difficult to check in practice.
In order to support this claim we provide in the appendix a simple proof at least in the easy case when $\rho$ is bounded from below and from above, see Lemma~\ref{lem:nablarho_is_gradient}.

Note finally that Equation \eqref{eq:chain_rule_G} can be stated for smooth curves of measures on a manifold, independently of the Riemannian metric. Fix a reference measure $\rd \nu$, which defines the entropy, then for $\mu_t = \rho_t \nu$:
\begin{equation}
    \frac{d}{dt} \H(\mu_t) = \int_M (\rd \log\rho_t) (F_t) \,.
\end{equation}
Indeed, also the continuity equation (in its weak formulation) can be written independently of the Riemannian structure. The extension to a larger class of curves of measures (as well as the definition of this class, and the optimal representative) needs however quantitative estimates, relying on the metric.

\section{Proofs and discussion}
\label{sec:proofs}
Our proofs below will strongly rely on K.T. Sturm's gradient estimate and Wasserstein contraction from \cite[Thm. 1.1]{sturm2025gradient}, which we recall here for convenience
\begin{align}
    &|\nabla P_t f|^2(x) \leq e^{4S \sqrt{t/\pi} + \mathcal O(t)} P_t(|\nabla f|^2)(x) \qquad \text{for all }f\in C^1(M), \label{eq:GE}\tag{GE}\\
    &W^2(P_t^*\mu,P_t^*\nu) \leq e^{4S \sqrt{t/\pi} + \mathcal O(t)}W^2(\mu,\nu) \qquad \text{for all }\mu,\nu\in \P(M). \label{eq:WC}\tag{WC}
\end{align}
By Kuwada duality \cite{kuwadaduality} both are actually equivalent to each other.
Here the meaning of $\mathcal O(t)$ is exactly as in our Theorem~\ref{theo:Fisher_decay}, namely uniform in $f,\mu,\nu$ for $t\in[0,t_0]$ sufficiently small.

The lower order term $\mathcal O(t)$ can be accessed explicitly in \cite{sturm2025gradient} and is essentially $+Kt$ for the Ricci-lower bound $\Ric(x)\geq -K$.
More precisely, when $M$ is convex ($S=0$) one can improve \eqref{eq:GE} to
\begin{equation}
\label{eq:GE_p=1}
|\nabla P_t f|(x) \leq e^{Kt} P_t(|\nabla f|)(x) \qquad \text{for all }f\in C^1(M)\text{ and } \ t\geq0
 \tag{GE\textsubscript{1}}
\end{equation}
see e.g. \cite[Cor.~3.2.6]{wang2014analysis}.

\begin{proof}[Proof of Theorem~\ref{theo:Fisher_decay_convex}]
This could be proved following a very indirect route in a purely metric setting as follows:
Convexity $S=0$ and Ricci lower bounds imply altogether that $\H$ is $\lambda$-displacement convex with $\lambda=-K$.
This in turn guarantees that $|\partial\H|^2=\I$ and that the heat flow coincides with the $\H$-gradient flow \cite{AGS1}.
It is moreover known that for $\lambda$-displacement convex functionals the slope decays exponentially, $|\partial\H|^2(\mu_t)\leq e^{-2\lambda t}|\partial\H|^2(\mu)$, which is exactly the claim.
For more details we refer to \cite[Thm.~2.4.15]{AGS}, \cite[Prop.~11.9]{AmbBruSem21}, see also \cite[pg.~97]{Top09}, \cite[Cor.~6]{Lot09} for similar inequalities in a Ricci flow background.

Instead, we provide here an elementary proof based on \eqref{eq:GE_p=1} and dispensing from displacement convexity arguments.
Take any $\mu=\rho V\in \P(M)$ with finite Fisher information (otherwise the statement is vacuous), and denote as always $\mu_t=P^*_t\mu=(P_t\rho)V$.
Then for $t>0$ we have that $\mu_t\in \P^\infty(M)$, i.e. $\mu_t=\rho_t V$ with $\rho_t$ smooth and bounded from above and from below.
Moreover, due to \eqref{eq:GE_p=1} and Jensen's inequality for $A(a,b)=\frac{|b|^2}{a}$ we have
\begin{multline*}
\I(\mu_t)=\int_M\frac{|\nabla P_t\rho|^2}{P_t\rho}\rd \vol
\leq e^{2Kt}\int_M\frac{(P_t|\nabla \rho|)^2}{P_t\rho}\rd \vol
\\
\leq e^{2Kt}\int_MP_t\left(\frac{|\nabla \rho|^2}{\rho}\right)\rd \vol
=e^{2Kt}\int_M\frac{|\nabla \rho|^2}{\rho}\rd \vol
=e^{2Kt}\I(\mu)
\end{multline*}
and the proof is complete.
\end{proof}
Note that a pointwise inequality as \eqref{eq:GE_p=1} is not available in the non-convex case.
This requires a workaround and makes the next proof more involved for $S>0$.
\begin{proof}[Proof of Theorem~\ref{theo:Fisher_decay}]

Assume first that $\mu=\rho V\in \P^\infty(M)$, thus making all computations below legitimate.
We give two independent proofs in that case, which we believe are of independent interest.

\paragraph{\textbf{Proof I (via the Wasserstein contraction)}}
Recall that we write $\mu_t=P^*_t\mu =(P_t\rho) V=\rho_t V$ for the Neumann heat flow.
By Wasserstein contraction \eqref{eq:WC} there holds, for $t>s\geq 0,h> 0$ small enough, 
\begin{equation}
\label{eq:W_contract}
    W^2(\mu_{t+h},\mu_t) \leq e^{4S \sqrt{(t-s)/\pi} + \mathcal O(t-s)}W^2(\mu_{s+h},\mu_s),
    \hspace{1cm}\forall\,0\leq s\leq t
\end{equation}
Noting that $\mu_t=\rho_t V\in \P^\infty(M)$ uniquely solves the continuity equation $\partial_t\mu_t+\dive(\mu_t v_t)=0$ with $v_t=-\nabla\log \rho_t$, Lemma \ref{lem:BBnon-convex} gives that
$\lim\limits_{h\to 0}\frac{W^2(\mu_{t+h},\mu_t)}{h^2}=\int_M|v_t|^2\rd \mu_t$ for a.e. $t\geq 0$.
Dividing \eqref{eq:W_contract} by $h^2$ and taking $h\to 0$ thus gives for a.e. $t>s>0$ sufficiently small
$$
    \int_M |\nabla \log \rho_t |^2 \rd \mu_t \leq e^{4S \sqrt{(t-s)/\pi} + \mathcal O(t-s)} \int_M|\nabla \log \rho_s|^2 \,\rd\mu_s\,.
$$
Since we assumed that $\rho\in C^{\infty}(M)$ is bounded from above and from below. so is $\rho_t$ for all $t>0$ (with bounds uniform in $t$).
In this simple setting an immediate application of Lebesgue's dominated convergence guarantees that $t\mapsto\int_M |\nabla \log \rho_t |^2 \rd \mu_t$ is continuous up to $t=0$.
Hence the previous inequality holds in fact for \emph{all} $0\leq s\leq t$ and the statement follows by taking $s=0$.
\paragraph{\textbf{Proof II (via the gradient estimate)}}
The following alternative proof is inspired by \cite[Ch.~5]{FH25}.
With the same notations as before, and leveraging the classical identity $(\nabla\log\rho_t)\rho_t =\nabla\rho_t$, there holds
    \begin{align*}
       \I(\mu_t) &= -\int_M|\nabla\log \rho_t|^2 \,\rd\mu_t + 2 \int_M|\nabla\log \rho_t|^2 \rho_t\,\rd V \\
       &= -\int_M|\nabla\log \rho_t|^2 \,\rd(P^*_t\mu) + 2 \int_M  \nabla\log \rho_t\cdot\nabla\rho_t \,\rd \vol \\
       &= -\int_MP_t\left( |\nabla\log \rho_t|^2\right)\, \rd\mu - 2 \int_M\Delta (\log \rho_t) \rho_t \,\rd \vol \\
       &= -\int_MP_t\left( |\nabla\log \rho_t|^2\right)\, \rd\mu - 2 \int_M\Delta (\log \rho_t)  \,\rd (P^*_t\mu) \\
       &= -\int_MP_t\left( |\nabla\log \rho_t|^2\right)\, \rd\mu - 2 \int_MP_t\left[\Delta (\log \rho_t) \right] \,\rd \mu\\
       &= -\int_MP_t\left( |\nabla\log \rho_t|^2\right)\, \rd\mu - 2 \int_M\Delta (P_t\log \rho_t) \,\rd \mu.
    \end{align*}
Here we used duality of $P_t,P^*_t$, integration by parts, and commutation $P_t\Delta=\Delta P_t$.
By the gradient estimate \eqref{eq:GE} with $f=\log\rho_t$ we can estimate the first term $P_t\left( |\nabla\log \rho_t|^2\right)\geq e^{-4S \sqrt{t/\pi} + \mathcal O(t)}|\nabla P_t\log\rho_t|^2$ in the right-hand side, and integrating by parts the second term we get
    \begin{align*}
      \I(\mu_t) 
       &\leq  -e^{-4S \sqrt{t/\pi} + \mathcal O(t)} \int_M|\nabla P_t\log \rho_t|^2 \rd\mu +2 \int_M \nabla P_t \log \rho_t\cdot \nabla \log \rho \,\rd\mu \\
       &\leq -e^{-4S \sqrt{t/\pi} + \mathcal O(t)} \int_M|\nabla P_t\log \rho_t|^2 \rd\mu
       \\
       & \hspace{1cm}+\left[e^{4S \sqrt{t/\pi} + \mathcal O(t)} \int_M |\nabla \log \rho|^2\rd\mu +e^{-4S \sqrt{t/\pi} + \mathcal O(t)} \int_M |\nabla P_t \log \rho_t|^2\rd\mu \right]
    \end{align*}
as claimed.
Here we used $\eps$-Young's inequality $ab\leq \frac 12(\eps a^2+\frac 1\eps b^2)$ with $\eps=e^{-4S \sqrt{t/\pi} + \mathcal O(t)}$.
\\

It only remains to settle the case of an arbitrary initial datum $\mu$.
Not surprisingly, we proceed by approximation.
Of course if $\I(\mu)=+\infty$ the statement is vacuous, so we may as well assume $\mu=\rho  V$, and by \eqref{eq:def_Fisher_H1_u} we can in fact assume that $u=\sqrt{\rho}\in H^1(M)$.
Let $u^n$ be any sequence converging strongly to $u$ in $H^1(M)$, with $u^n$ smooth and bounded away from zero (take for example $u^n=P_{1/n}u$), and normalize $\rho^n=\frac{|u^n|^2}{\|u^n\|^2_{L^2}}$ so that $\mu^n=\rho^n V\in \P^\infty(M)$.
Owing to the strong $H^1$ convergence $u^n\to u$ we have that $\rho^n=|u^n|^2\to |u|^2= \rho$ strongly in $L^1$, with $\|u^n\|^2_{L^2}\to \|u\|^2_{L^2}=\|\rho\|_{L^1}=1$.
Whence
$$
\I(\mu^n)=4\int_M|\nabla \sqrt{\rho^n}|^2\rd V=\frac 4{\|u^n\|^2_{L^2}}\int_M|\nabla u^n|^2\rd V
\xrightarrow[n\to\infty]{}4\int |\nabla u|^2\rd V=\I(\mu).
$$
Let now $\mu^n_t=P^*_t\mu^n=(P_t\rho^n) V$, see Figure~\ref{fig:mu_nt}.
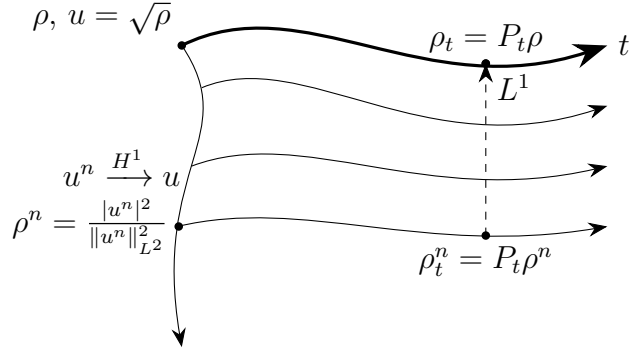
\begin{figure}[h!]
    \centering
\begin{tikzpicture}[scale=0.8]

  \coordinate (Rho0) at (0,5);
  \coordinate (RhoT) at (7,5);
  \coordinate (Rho0N) at (0,0);
  \coordinate (RhoTN) at (7,0);
  
  \coordinate (Rho0n) at (0.33,4.3);
  \coordinate (Rhotn) at (7,4);

  \coordinate (Rho0n2) at (0.16,3);
  \coordinate (Rhotn2) at (7,3);

  \coordinate (Rho0n3) at (-0.07,2);
  \coordinate (Rhotn3) at (7,2);
  
  \coordinate (Rhotn33) at (5,1.85);
  \coordinate (Rhot) at (5,4.7);
  
  \fill (Rho0) circle (2pt) node[above left] {$\rho$, $u=\sqrt{\rho}$};
  \fill (Rho0n2) node[left] {$u^n\xrightarrow[]{H^1} u$};
  \fill (-.05,2) circle (2pt) node[left] {$\rho^n=\frac{|u^n|^2}{\|u^n\|^2_{L^2}}$};
  \fill (Rhotn33) circle (2pt) node[below] {$\rho_t^n = P_t \rho^n$};
  \fill (Rhot) circle (2pt) node[above] {$\rho_t = P_t \rho$};
  \fill (Rhot) circle (2pt) node[below right] {$L^1$};
  \node[right] at (RhoT) {$t$};
  
  \draw[->, >={Stealth[scale=1.5]}, very thick] (Rho0) .. controls (2,6) and (4,4) .. (RhoT);
  \draw[->, >={Stealth[scale=1.5]}] (Rho0) .. controls (1,3.5) and (-0.5,3) .. (Rho0N);

  \draw[->, >={Stealth[scale=1.5]}] (Rho0n) .. controls (2,5) and (4,3) .. (Rhotn);
  \draw[->, >={Stealth[scale=1.5]}] (Rho0n2) .. controls (2,3.7) and (4,2.3) .. (Rhotn2);
  \draw[->, >={Stealth[scale=1.5]}] (Rho0n3) .. controls (2,2.5) and (4,1.5) .. (Rhotn3);
  \draw[->, >={Stealth[scale=1.5]}, dashed] (Rhotn33) to (Rhot);
  

\end{tikzpicture}

    \caption{flow approximation $\mu^n_t=(P_t\rho^n)\vol$}
    \label{fig:mu_nt}
\end{figure}

\noindent
By the first step there holds
$$
\I(\mu^n_t)\leq e^{4S\sqrt{t/\pi}+\mathcal O(t)}\I(\mu^n)
$$
for small $t\geq 0$.
By the Wasserstein contraction \eqref{eq:WC} and because $\rho^n\to \rho$ in $L^1$ $\implies$ $\mu^n\narrowcv\mu$ narrowly $\implies$ $W(\mu^n,\mu)\to 0$, we get
$$
W^2(\mu^n_t,\mu_t)\leq e^{4S\sqrt{t/\pi}+\mathcal O(t)}W^2(\mu^n,\mu) \xrightarrow[n\to\infty]{}0
$$
and in particular $\mu^n_t\narrowcv \mu_t$ for any fixed $t$.
By lower semi-continuity in the previous inequality and $\I(\mu^n)\to \I(\mu)$ we get
$$
\I(\mu_t)
\leq
\liminf\limits_{n\to\infty}\I(\mu^n_t)
\leq
\liminf\limits_{n\to\infty}e^{4S\sqrt{t/\pi}+\mathcal O(t)}\I(\mu^n)
=
e^{4S\sqrt{t/\pi}+\mathcal O(t)}\I(\mu)
$$
and the proof is complete.
\end{proof}

\begin{proof}[Proof of Theorem~\ref{theo:fisher_upper_grad}]
Since the statement is invariant under time-reparametrization and any AC curve can be reparametrized into a constant-speed Lipschitz curve (using arclength, see e.g. \cite[Lem. 1.1.4]{AGS}) we might as well assume that $\mu \in AC_2([0,T];W)$.

The classical idea is to construct an $\ed$-regularization $\mu^\ed_t$, where $\eps$ and $\delta$ control the space and time regularizations, respectively, justify an exact chain-rule for $\ed>0$, and then pass to the limit first as $\delta\to 0$ and then $\eps\to 0$ with good estimates.
The main caveat is that, in Riemannian subdomains and in contrast with whole Euclidean spaces, the space regularization cannot simply be achieved using standard convolution, under which the kinetic action and Fisher information would both be well-behaved by classical convexity arguments (in particular this requires spatial derivatives and convolution to commute).
We use instead the Neumann heat flow as a regularization procedure, but this poses significant technical challenges because $\nabla$ and $P_t$ do not commute in general.
Roughly speaking, the kinetic energy will be controlled via Sturm's gradient estimate and Wasserstein contraction  \eqref{eq:GE}\eqref{eq:WC}, while the Fisher information will be controlled as in Theorem~\ref{theo:Fisher_decay}.
\paragraph{\underline{Step 1:} regularization}
For small $\eps>0$ we define
\begin{equation*}
\mu^\eps_t\coloneqq P^*_\eps\mu_t,
\end{equation*}
i.e. the solution at time $\eps$ of the Neumann-heat flow started from $\mu_t$ as initial datum.
By \eqref{eq:WC} one has that $\mu^\eps\in AC([0,T];W)$ with explicit control of the metric speed
\begin{equation}
\label{eq:speed_mu_eps}
|\dot\mu^\eps_t|
\leq e^{2S\sqrt{\eps/\pi}+\mathcal O(\eps)}|\dot\mu_t|
\leq e^{C\sqrt\eps}|\dot\mu_t|
\end{equation}
for some $C>0$ only depending on the geometry of the domain (through our Ricci and convexity lower bounds $K,S<+\infty$).
Moreover, by properties of the heat flow \cite{grigoryan2009heat,wang2014analysis,wang_gradient} the density $\rho^\eps_t(x)$ of $\mu^\eps_t=\rho^\eps_t \vol$ is smooth up to the boundary with
\begin{equation}
\label{eq:heat-kernel-bounds}
|\nabla\rho^\eps_t(x)|\leq C_\eps
\qqtext{and}
c_\eps \leq \rho^\eps_t(x)\leq C_\eps
\end{equation}
for some finite $c_\eps,C_\eps>0$ depending only on the domain and $\eps>0$, uniformly in $t$ .

By the Benamou-Brenier formula (Lemma \ref{lem:BBnon-convex}) there is a velocity field $v^\eps=v^\eps_t(x)$ optimally representing the curve $\mu^\eps$ in the sense that
$$
\partial_t\mu^\eps+\dive(\mu^\eps v^\eps)=0
\qqtext{with}
\int_M |v^\eps_t|^2\rd\mu^\eps_t = |\dot\mu^\eps_t|^2
\quad\text{for a.e. }t\in (0,T).
$$
We denote the corresponding momentum as
$$
F^\eps= \mu^\eps v^\eps.
$$
Note that, since $\mu^\eps$ has smooth density, $F^\eps$ is also absolutely continuous in space.
As a consequence we do not distinguish below the measure-momentum and its density (w.r.t $V$), and simply write $F^\eps$.

Next we regularize in time.
To this end we first extend by continuity $\mu^\eps_t=\mu^\eps_0$ for $t\leq 0$ and $\mu^\eps_t=\mu^\eps_T$ for $t\geq T$, and also set $F^\eps_t=0$ for $t\not\in[0,T]$.
These extended $(\mu^\eps,F^\eps)$ solve the continuity equation for $t\in \R$.
If $\eta^\delta(t)=\frac 1\delta\eta(t/\delta)$ is a standard one-dimensional mollifier supported in $[-\delta,\delta]$, we take the time-convolution
$$
\mu^\ed_t\coloneqq \eta^\delta\ast \mu^\eps_t
\qqtext{and}
F^\ed_t\coloneqq \eta^\delta\ast F^\eps_t.
$$
Since the time convolution commutes with spatial derivatives, one still has
$$
\partial_t\mu^\ed+\dive F^\ed=0
$$
for $t\in\R$.
As before, $F^\ed$ is absolutely continuous in space and we identify the measure and its density with a slight abuse of notations.
\paragraph{\underline{Step 2:} chain rule}
By usual properties of the convolution and $\nabla \rho^\ed_t=\nabla(\eta^\delta\ast \rho^\eps_t)=\eta^\delta\ast(\nabla \rho^\eps_t)$ it is immediate to see that $\rho^\ed_t$ satisfies the exact same bounds as \eqref{eq:heat-kernel-bounds}, namely
\begin{equation}
\label{eq:heat-kernel-bounds-delta}
\forall\,t\in\R,\,x\in\O:\qquad |\nabla\rho^\ed_t(x)|\leq C_\eps
\qqtext{and}
0<c_\eps \leq \rho^\ed_t(x)\leq C_\eps
\end{equation}
uniformly in $\delta$.
In particular, $\rho^\ed$ is smooth in time and space (up to the boundary), bounded from above and from below, and satisfies the continuity equation with no-flux boundary condition.
This is more than enough to differentiate under the integral sign, viz
\begin{equation}
\label{eq:chain_rule_eps_delta}
\H(\mu^\ed_{t_1})-\H(\mu^\ed_{t_0})
=\int_{t_0}^{t_1}\int_M \nabla \log\rho^\ed_t \cdot F^\ed_t\,\rd\vol\,\rd t.
\end{equation}
\paragraph{\underline{Step 3:} limit $\delta\to 0$.}
As $\delta\to 0 $ we have of course that $\rho^\ed\to\rho^\eps$ and $F^\ed\to F^\eps$ at least for a.e. $(t,x)$.
For fixed $\eps$, the density $\rho^\ed$ is bounded from above and from below uniformly in $t,x$ but most importantly uniformly in $\delta$.
This suffices to apply Lebesgue's dominated convergence for $\H(\mu^\ed_{t})\to \H(\mu^\eps_t)$ for all $t$, so the left-hand side of \eqref{eq:chain_rule_eps_delta} immediately passes to the limit.
For the right-hand side, we have similarly that $\nabla\log\rho^\ed=\frac{\nabla\rho^\ed}{\rho^\ed}\to \frac{\nabla \rho^\eps}{\rho^\eps}$ and $F^\ed\to F^\eps$ at least pointwise almost everywhere.
By Vitali's convergence theorem it is enough to check that $f^\ed\coloneqq \nabla \log\rho^\ed \cdot F^\ed$ is uniformly integrable.
For this we simply leverage \eqref{eq:heat-kernel-bounds-delta} to bound $|\nabla\log\rho^\ed|=\left|\frac{\nabla\rho^\ed}{\rho^\ed}\right|\leq C_\eps$ as well as $|F^\ed|=\sqrt{\rho^\ed}\left|\frac{F^\ed}{\sqrt{\rho^\ed}}\right|\leq C_\eps\left|\frac{F^\ed}{\sqrt{\rho^\ed}}\right|$.
By Jensen's inequality (for the convex function $A(a,b)=\frac{|b|^2}{a}$) we estimate
\begin{multline*}
\int_\R\int_M|f^\ed|^2\,\rd\vol\,\rd t
\leq
C_\eps\int_\R\int_M \frac{|F^\ed|^2}{\mu^\ed}\,\rd t
=
C_\eps\int_\R\int_M \frac{|\eta^\delta\ast F^\eps|^2}{\eta^\delta\ast \mu^\eps}\,\rd t
\\
\leq
C_\eps\int_\R\int_M \eta^\delta\ast\left(\frac{| F^\eps|^2}{\mu^\eps}\right)\,\rd t
=
C_\eps \int_\R\eta^\delta\ast \left(\int_M \frac{|F^\eps|^2}{\rho^\eps}\right)\rd t
\\
= C_\eps\int_{\R}\int_M\frac{|F^\eps|^2}{\rho^\eps}\rd t
=C_\eps \int_0^T |\dot\mu^\eps_t|^2\rd t
\end{multline*}
uniformly in $\delta$.
Owing to \eqref{eq:speed_mu_eps} and our standing assumption that $\mu\in AC_2$ we have that $\int_0^T |\dot\mu^\eps_t|^2\rd t\leq e^{C\eps}\int_0^T |\dot\mu_t|^2\rd t$ is finite.
The family $\{f^{\ed}\}_{\delta>0}$ is therefore uniformly integrable and we can take the limit $\delta\to 0$ in \eqref{eq:chain_rule_eps_delta} as
\begin{equation}
\label{eq:chain_rule_eps}
\H(\mu^\eps_{t_1})-\H(\mu^\eps_{t_0})
=
\int_{t_0}^{t_1}\int_M \nabla \log\rho^\eps_t \cdot F^\eps_t\, \rd\vol\,\rd t
=
\int_{t_0}^{t_1}\int_M \nabla \log\rho^\eps_t \cdot  v^\eps_t\, \rd\mu^\eps_t\,\rd t \,.
\end{equation}
\paragraph{\underline{Step 4:} limit $\eps\to 0$}
Applying the Cauchy-Schwarz inequality in $L^2(\mu_t)$ for a.e. $t$ we obtain at once
$$
\left|\H(\mu^\eps_{t_1})-\H(\mu^\eps_{t_0})\right|
\leq 
\int_{t_0}^{t_1}\|\nabla\log\rho^\eps_t\|_{L^2(\mu^\eps_t)}\left\|v^\eps_t\right\|_{L^2(\mu^\eps_t)}\rd t
=\int_{t_0}^{t_1}\sqrt\I(\mu^\eps_t)|\dot\mu^\eps_t|\rd t,
$$
since $v^\eps$ was precisely chosen to optimally represent $\mu^\eps$ in the sense of Lemma~\ref{lem:BBnon-convex}.
By Theorem~\ref{theo:Fisher_decay} and \eqref{eq:speed_mu_eps} we see that the right-hand side can be controlled as
$$
\left|\H(\mu^\eps_{t_1})-\H(\mu^\eps_{t_0})\right|
\leq 
e^{C\sqrt\eps}\int_{t_0}^{t_1}\sqrt\I(\mu_t)|\dot\mu_t|\rd t,
$$
and we wish to pass to the limit $\eps\to 0$.
The right-hand side is trivially converging.
For the left-hand side we claim that $\H(\mu^\eps_t)\to\H(\mu_t)$ for any $t$.
Indeed, by Jensen's inequality one has $\H(\mu^\eps_t)\leq \H(\mu_t)$ for any $t$.
But by lower semicontinuity of the entropy we also have $\H(\mu_t)\leq \liminf\limits_{\eps\to 0}\H(\mu^\eps_t)$, which entails the claim.
Passing to the limit in the above inequality finally establishes \eqref{eq:chain_rule_upper_grad} and concludes the proof.
\end{proof}

\begin{proof}[Proof of Theorem~\ref{theo:chain_rule_G}]
Owing to our assumption \eqref{eq:finite_energy_G} and the Benamou-Brenier formula (Lemma~\ref{lem:BBnon-convex}), we see that $\mu\in AC_2([0,T];W)$ with moreover
$$
|\dot\mu_t|^2\leq \int_M\frac{|G_t|^2}{\mu_t}
\qquad \text{for a.e. }t\in (0,T).
$$
We then proceed exactly as in the proof of Theorem~\ref{theo:fisher_upper_grad}: the regularized curve $\mu^\eps_t=P^*_\eps \mu_t$ is still $AC_2$ with
$$
|\dot\mu^\eps_t|\leq e^{C\sqrt\eps}|\dot\mu_t|,
$$
and one can pick an optimal representative $F^\eps=\mu^\eps v^\eps$ satisfying
$$
\partial_t\mu^\eps_t +\dive F^\eps_t =0
\qqtext{with}
\int_M \frac{|F^\eps_t|^2}{\mu^\eps_t}=|\dot\mu^\eps_t|^2\text{ for a.e. }t\in (0,T).
$$
Further regularizing in time and arguing exactly as in the previous proof, we end-up with
\begin{equation}
\label{eq:chain_rule_eps_G}
\H(\mu^\eps_{t_1})-\H(\mu^\eps_{t_0})
=
\int_{t_0}^{t_1}\int_M \nabla \log\rho^\eps \cdot F^\eps
=
\int_{t_0}^{t_1}\int_M \nabla \log\rho^\eps_t \cdot v^\eps_t\,\rd \mu^\eps_t\rd t
\end{equation}
as before.
Note that the right hand side lies in $L^1(0,T)$ for any fixed $\eps>0$.
Indeed, writing for convenience $\bar\mu^\eps=\mu^\eps_t\rd t$ for the natural space-time measure, our assumptions guarantee that $\nabla\log\rho^\eps$ and $v^\eps$ both belong to $L^2(\bar \mu^\eps)$, with quantitative control
\begin{equation}
\label{eq:quantitative_control_I_mu_eps}
\|\nabla\log \rho^\eps\|^2_{L^2(\bar \mu^\eps)}
=\int_0^T\int_{\O}|\nabla\log\rho^\eps_t|^2\,\rd\mu^\eps_t\rd t
=\int_0^T \I(\mu^\eps_t)\rd t
\leq e^{C\sqrt{\eps}}\int_0^T \I(\mu_t)\rd t
\end{equation}
and
\begin{equation}
\label{eq:quantitative_control_v_eps}
\|v^\eps\|^2_{L^2(\bar \mu^\eps)}
=\int_0^T\int_{\O}|v^\eps_t|^2\,\rd\mu^\eps_t\rd t
=\int_0^T |\dot\mu^\eps_{t}|^2\rd t
\leq e^{C\sqrt{\eps}}\int_0^T |\dot\mu_t|^2\rd t.
\end{equation}
The goal is now to pass to the limit $\eps\to 0$ in \eqref{eq:chain_rule_eps_G}.
Note that the left-hand side passes to the limit exactly as before.
To handle the right-hand side recall first that $\int_0^T|\dot\mu^\eps_t|^2\rd t\leq C$ uniformly in $\eps$.
The standard Arzel\`a-Ascoli theorem therefore guarantees that $\mu^\eps\to \mu$ uniformly in $C([0,T];W)$.
As a consequence the space-time measure $\bar\mu^\eps\narrowcv \bar\mu\coloneqq \mu_t\rd t$ converges narrowly in $\M([0,T]\times M)$.
By rather standard compactness arguments \cite[Thm.~5.4.4]{AGS}, \eqref{eq:quantitative_control_I_mu_eps}\eqref{eq:quantitative_control_v_eps} show that, up to discrete subsequences if needed, $u^\eps\coloneqq \nabla\log\rho^\eps$ and $v^\eps$ both converge weakly to some $u,v\in L^2(\bar\mu)$ in the sense that
\begin{align}\label{eq:weakconvAGS}
\int_0^T\int_M u^\eps \cdot\xi\,\rd\bar\mu^\eps
\to \int_0^T\int_M u \cdot\xi\,\rd\bar\mu
\ \ \  \text{and} \ \ \
\int_0^T\int_M v^\eps \cdot\xi\,\rd\bar\mu^\eps
\to \int_0^T\int_M v \cdot\xi\,\rd\bar\mu
\end{align}
for all $\xi\in C^\infty_c([0,T]\times M)$, with moreover
\begin{align} \label{eq:lowersc}
\|u\|_{L^2(\bar\mu)} \leq \liminf\limits_{\eps\to 0}\|u^\eps\|_{L^2(\bar\mu^\eps)}
\qqtext{and}
\|v\|_{L^2(\bar\mu)} \leq \liminf\limits_{\eps\to 0}\|v^\eps\|_{L^2(\bar\mu^\eps)}.
\end{align}
This means in particular that $\mu^\eps u^\eps \to \mu u$ in the sense of distributions, and since $\mu^\eps u^\eps=\nabla\mu^\eps$ it is easy to check that necessarily $u$ is the measure-theoretic logarithmic derivative $u=\nabla\log\rho$ in the sense of \eqref{eq:log_derivative}.
Moreover, owing to \eqref{eq:quantitative_control_I_mu_eps} we also have the upper bound
$
\limsup \limits_{\eps\to 0}\|u^\eps\|^2_{L^2(\bar\mu^\eps)}\leq \int_0^T\I(\mu_t)\rd t$.
This implies
\begin{align}\label{eq:conv_of_nablalog}
\lim\limits_{\eps\to 0}\|u^\eps\|^2_{L^2(\bar\mu^\eps)}=\|u\|^2_{L^2(\bar\mu)}
\end{align}
and means that $u^\eps\to u$ strongly in the sense of \cite[Def.~5.4.3]{AGS}.

With now $u^\eps\to u$ strongly and $v^\eps\rightharpoonup v$ weakly we can apply a weak-strong convergence procedure, Lemma~\ref{lem:weak_strong_CV} in the Appendix, to guarantee that
$$
\int_{t_0}^{t_1}\int_M \nabla \log\rho^\eps_t  \cdot v^\eps_t\,\rd \mu^\eps_t\rd t
\xrightarrow[\eps\to 0]{}
\int_{t_0}^{t_1}\int_M \nabla \log\rho_t  \cdot v_t\,\rd \mu_t\rd t
=\int_{t_0}^{t_1}\int_M \nabla \log\rho \cdot \rd F,
$$
where we put $F=v_t\rd\mu_t\rd t$.
By \eqref{eq:chain_rule_eps_G} we conclude that
$$
\H(\mu_{t_1})-\H(\mu_{t_0})
=
\int_{t_0}^{t_1}\int_M \nabla \log\rho \cdot \rd F,
\hspace{1cm}\forall\ 0<t_0\leq t_1<T.
$$

Moreover, weak convergence $v^\eps\rightharpoonup v$ implies that $F^\eps=\mu^\eps v^\eps \to \mu v=  F$, at least in the sense of distributions. Since $\partial_t\mu^\eps+\dive( \mu^\eps v^\eps)=0$ we see that the pair $(\mu,F)$ solves the continuity equation $\partial_t\mu+\dive F=0$ (with no-flux boundary conditions), just as $(\mu,G)$.

Now, for the first statement, it is left to show that $F$ is the optimal representative of $\mu$. This follows by lower semicontinuity \eqref{eq:lowersc} and Equation \eqref{eq:quantitative_control_v_eps}:
\begin{align}
    \|v\|^2_{L^2(\bar \mu)} \leq \liminf_{\eps\to0}\|v^\eps\|^2_{L^2(\bar \mu^\eps)} \leq \liminf_{\eps\to 0} e^{C\sqrt{\eps}}\int_0^T |\dot\mu_t|^2\rd t = \int_0^T |\dot\mu_t|^2\rd t. 
\end{align}

For the second statement, writing $F=\mu v,G=\mu w$ for $v,w\in L^2(\mu)$, with
$$
\int_M \nabla\phi\cdot\rd F_t=
\int_M \nabla\phi\cdot v_t\,\rd\mu_t 
=
\int_M \nabla\phi\cdot w_t\,\rd\mu_t
=\int_M \nabla\phi\cdot\rd G_t,
\hspace{1cm}\forall\,\phi\in C^1(M)
$$
for a.e. $t$.
By density this holds with $\nabla\phi$ replaced by any vector-field $\xi$ in the {$L^2(\mu_t)$-closure} of $\{\nabla\phi,\,\phi\in C^1(M)\}$.
By assumption $u_t=\nabla\log\rho_t$ is such a velocity field, and as a consequence
$$
\H(\mu_{t_1})-\H(\mu_{t_0})
=
\int_{t_0}^{t_1}\int_M \nabla \log\rho \cdot \rd F
=
\int_{t_0}^{t_1}\int_M \nabla \log\rho \cdot \rd G
$$
as in our statement.
Finally, let us check that $t\mapsto \H(\mu_{t_1})$ is absolutely continuous, which amounts to checking that $t\mapsto \int_M \nabla \log\rho _t\cdot \rd G_t$ is $L^1$.
Given that $G=\mu w$ with $w_t\in L^2(\mu_t)$ we simply observe by Cauchy-Schwarz inequality that
\begin{multline*}
    \int_0^T\left|\int_M\nabla\log\rho_t \cdot \rd F_t\right|\rd t
    =\int_0^T\left|\int_M\nabla\log\rho_t \cdot v_t\rd \mu_t\right|\rd t
    \\
    \leq
    \left(\int_0^T\int_M|\nabla\log\rho_t|^2\rd\mu_t\rd t\right)^\frac12
     \left(\int_0^T\int_M|v_t|^2\rd\mu_t\rd t\right)^\frac12
     \\
    =\left(\int_0^T\I(\mu_t)\rd t\right)^\frac12
    \left(\int_0^T\int_M\frac{|F_t|^2}{\mu_t}\rd t\right)^\frac12<+\infty
\end{multline*}
and the proof is complete.
\end{proof}

\section{Extension to other entropies}\label{sec:general_ent}
We now extend the main results to more general entropies. Given a potential $U\in C^1(M)$, define
\begin{align}
    \E_{m,U}(\mu)= \begin{cases}
        \int f_m(\rho) \rd\vol + \int U \rd \mu \qquad & \text{for $\mu=\rho\vol$} \\
        +\infty &\text{otherwise}
    \end{cases}
\end{align}
with
\begin{align}
    f_m(x)=\begin{cases}
        x\log x &\text{if $m=1$} \\
        \frac{x^m}{m-1} &\text{if $m \neq 1$.}
    \end{cases}
\end{align}
The Fisher information functional is replaced by
\begin{align}
    \I_m (\mu) = \int |\nabla g_{m,U}(\rho)|^2 \rd\mu
\end{align}
where
\begin{align}
    g_{m,U}(\rho) = f_m'(\rho) + U =\begin{cases}
            \log\rho + U  \qquad &\text{if $m=1$} \\
            \frac{m}{m-1}\rho^{m-1} + U &\text{else.}
        \end{cases}
\end{align}
More precisely, $\I_m$ is the right-hand side above when $\mu$ has logarithmic  in $L^2(\mu)$ for $m=1$, or if $\rho^m\in W^{1,1}(\Omega)$, with $\nabla\rho^m= u\rho$ and $u\in L^2(\mu)$ (in this case we formally write $u=\frac{m}{m-1}\nabla \rho^{m-1}$) for $m\neq 1$, and $\I_m (\mu)=+\infty$ otherwise.

\begin{theo}
    Let $m\in [1,3/2]$. Then:
    \begin{enumerate}[(i)]
        \item The functional $\sqrt{\I_m}$ is a strong Wasserstein upper gradient for $\E_{m,U}$. 
        \item If $\mu=(\mu_t)_{t\in[0,T]}$ has finite kinetic energy and $\int_0^T \I_m(\mu_t)dt < +\infty$, then
    \begin{align}
        \frac{d}{dt}\E_{m,U}(\mu_t)=\int_M \nabla g_{m,U}(\rho) \cdot \rd F_t \,,
    \end{align}
    where $F_t$ is the optimal representative of $(\mu_t)_{t\in[0,T]}$ the sense of Lemma \ref{lem:BBnon-convex}.
    \end{enumerate}
\end{theo}

\begin{proof} 
We first treat the case $m>1$, $U\equiv0$, and then see how to deal with the additional potential term.
\paragraph{\underline{Step 1:} the case $1< m \leq 3/2$}
Firstly note that, for $m\in (1,3/2]$, one can choose $p>1$ such that the function on $[0,\infty)\times [0,\infty)$
\begin{align}
(a,b)\mapsto 
\begin{cases}
    a^{2m-3}b^{2/p} & \text{if $a>0$} \\
    0 & \text{if $a=b=0$} \\
    +\infty & \text{else}
\end{cases}    
\end{align} 
is jointly convex.

Setting $U\equiv 0$, arguing as in \cite[Rmk. 5.9]{erbar2024gradient}, we rewrite the Fisher information as
$$
\I_m(\mu)=\int m^2\rho^{2m-3}(|\nabla\rho|^{p})^{2/p}\,\rd\vol,
$$
where $\rho^{2m-3}(|\nabla\rho|^{p})^{2/p}$ is understood as above.
Using the estimate from \cite[Thm.~1.1(iii)]{sturm2025gradient} the heat flow satisfies the pointwise gradient estimate
$$ 
|\nabla P_t\rho|^p(x)
\leq
e^{C\sqrt{t}} P_t(|\nabla\rho|^p)(x),
$$
whence by Jensen's inequality
\begin{align}\label{eq:Im_control}
\frac{\I_m(\mu_ t )}{m^2}
  \leq 
\int (P_ t \rho)^{2m-3}\left[e^{C\sqrt t } P_ t (|\nabla\rho|^ p)\right]^{2/ p}\,\rd\vol 
 &\leq  e^{C\sqrt t }
\int P_ t \left[\rho^{2m-3}(|\nabla\rho|^{ p})^{2/ p}\right]\,\rd\vol
\\
 &\leq   e^{C\sqrt t }
\int\rho^{2m-3}(|\nabla\rho|^{ p})^{2/ p}\,\rd\vol
 = e^{C\sqrt t } \frac{\I_m(\mu)}{m^2},
\end{align}
where the constant $C>0$ in the exponential may vary from line to line but is uniform in $t,\mu$. 

(i) This control of $\I_m$ along the heat flow allows one to repeat the proof of Theorem \ref{theo:fisher_upper_grad}: regularizing $(\mu_t)_{t\in[0,T]}$ again using the heat flow, Equation \eqref{eq:chain_rule_eps_delta} becomes
\begin{align}
    \E_m(\mu_{t_1}^{\eps,\delta}) - \E_m(\mu_{t_0}^{\eps,\delta}) = \frac{m}{m-1}\int_{t_0}^{t_1} \int_M \nabla (\rho_t^{\eps,\delta})^{m-1} \cdot F_t^{\eps,\delta} \rd\vol \rd t \,.
\end{align}
Sending $\delta\to 0$ follows analogously, while for $\eps\to 0$ we crucially use the control \eqref{eq:Im_control}, lower-semicontinuity of $\E_m$, and convexity of $x\mapsto x^m$.

(ii) For the exact chain rule, we apply the same compactness arguments to get $\bar \mu^\eps \narrowcv \mu_t\rd t$, $\nabla(\frac{m}{m-1}(\rho^\eps)^{m-1})=\nabla g_m(\rho^\eps) =: u^\eps \rightharpoonup u$ and $v_\eps \rightharpoonup v$ exactly as in \eqref{eq:weakconvAGS}. The only more involved step is the identification of the limit $u = \nabla g_m(\rho)$ (in the formal sense defined above). Since $\rho^\eps$ are smooth and positive, $\rho^\eps u^\eps = \nabla(\rho^\eps)^m$ and for every smooth $\xi$ test
\begin{align}\label{eq:weakderiv_m}
    \int\int u^\eps \cdot \xi \ \rd\bar\mu^\eps = -\int_0^T\int_M (\rho_t^\eps)^m \dive\xi \ \rd\vol \ \rd t \,.
\end{align}
The LHS passes to the limit by convergence in distributional sense of $u^\eps d\bar\mu^\eps$, while for the RHS we argue as follows.
Owing to our standing assumption that $\I_m(\mu) \in L^1(0,T)$ (i.e. $\nabla \rho^{m-\frac{1}{2}}\in L^1_tL^2_x$), an application of Jensen's (J) and Gagliardo-Nirenberg's (GN) inequalities gives that
\begin{multline*}
    \int_0^T \int_M (\rho_t^\eps)^{2m-1} \rd\vol \rd t 
    \overset{\text{J}}{\leq }
    C \int_0^T \left(\int_M (\rho_t^\eps)^{(2m-1)\frac{d}{d-2}} d\vol\right)^{\frac{d-2}d}\rd t 
    \\
    =C \int_0^T \|\left(\rho^\eps_t\right)^{m-\frac 12}\|^2_{L^{2^*}(M)} \rd t
    \\
     \overset{\text{GN}}{\leq} C\int_0^T (1+ \I_m(\mu_t^\eps))\rd t 
    \overset{\eqref{eq:Im_control}}{\leq} C\int_0^T (1+ \I_m(\mu_t))\rd t < \infty
\end{multline*}
uniformly in $\eps>0$.
Since $2m-1>m$ as soon as $m>1$ this gives space-time equiintegrability for $\left\{\left(\rho^\eps\right)^{m}\right\}_{\eps>0}$, and by Vitali's convergence theorem we conclude that the RHS of \eqref{eq:weakderiv_m} converges.

\paragraph{\underline{Step 2:} the potential term} For the sake of notation we work with the case $m=1$, the $m\neq 1$ case follows in the same way. 

(i) Equation \eqref{eq:chain_rule_eps_delta} is now
\begin{align}\label{eq:ecr_drifted_ed}
    \E_{1,U}(\mu^{\ed}_{t_1})- \E_{1,U}(\mu^{\ed}_{t_0}) = \int_{t_0}^{t_1} \int_M (\nabla\log \rho_t^{\ed} + \nabla U) \cdot F_t^{\ed} \rd\vol \rd t\,.
\end{align}
Passing to $\delta\to 0$ is analogous, and after Cauchy--Schwarz we obtain
\begin{align}
    \left| \E_{1,U}(\mu^{\eps}_{t_0})- \E_{1,U}(\mu^{\eps}_{t_1}) \right| &\leq  \int_{t_0}^{t_1} \left( \int_M |\nabla\log \rho_t^{\eps} + \nabla U|^2 \rd\mu_t^\eps \right)^{1/2} |\dot\mu_t^\eps| \rd t \\
    &\leq e^{C\sqrt\eps} \int_{t_0}^{t_1} \left( \int_M (|\nabla\log \rho_t^{\eps}|^2 + 2\nabla\log \rho_t^{\eps} \cdot\nabla U + |\nabla U|^2) \rd\mu_t^\eps \right)^{1/2} |\dot\mu_t| \rd t
\end{align}
Now, in order to pass to the limit we will use Lebesgue dominated convergence on $[t_0,t_1]$. We have $\int_M |\nabla\log \rho_t^{\eps}|^2 \rd \mu_t^\eps \leq e^{C\sqrt\eps} \I(\mu_t)$ as in Theorem \ref{theo:fisher_upper_grad}, $\int_M |\nabla U|^2\rd\mu_t^\eps \to \int_M |\nabla U|^2\rd\mu_t$ by weak convergence of measures, and moreover $\int_M |\nabla U|^2\rd\mu_t^\eps  \leq \lVert \nabla U \rVert_{C^0}^2$. Finally the mixed term is upper bounded by Young's inequality and previous estimates, and converges to $\int 2\nabla\log\rho \cdot \nabla U d\mu_t$ by applying Lemma \ref{lem:weak_strong_CV} with $v_n = \nabla U$, $u_n=u_{1/\eps}=\nabla\log\rho_t^\eps$. The integrand is then uniformly bounded:
\begin{align}
    &\left( \int_M (|\nabla\log \rho_t^{\eps}|^2 + 2\nabla\log \rho_t^{\eps} \cdot\nabla U + |\nabla U|^2) \rd\mu_t^\eps \right)^{1/2} |\dot\mu_t| \\
    &\leq \left( 2e^{C\sqrt\eps} \I(\mu_t) + 2\lVert\nabla U\rVert_{C^0}^2 \right)^{1/2} |\dot\mu_t| \\
    &\leq C(1+\sqrt\I(\mu_t)) |\dot\mu_t| \label{eq:boundImudot}
\end{align}
by using $\sqrt{x^2+C}\leq C'(1+x)$. The bound \eqref{eq:boundImudot} is integrable if and only if $\int_{t_0}^{t_1}\sqrt\I(\mu_t) |\dot\mu_t|\rd t$ is finite, while if it is not the upper gradient estimate is trivial.

(ii) The proof is completely analogous, noting that the drift term from \eqref{eq:ecr_drifted_ed} easily passes to the limit, by definition of $v^\eps\rightharpoonup v$.
\end{proof}

Observe that the same limitation $m\leq 3/2$ appears in \cite{erbar2024gradient} for the exact same reason that $\I_m$ must be jointly convex in $(\rho,\nabla\rho)$. It would be interesting to extend the result to $m > 3/2$ or more general, non algebraic Fisher-like functionals, when the joint convexity fails.

In both our chain rules, Theorem~\ref{theo:fisher_upper_grad} and Theorem~\ref{theo:chain_rule_G}, the statements themselves are independent of how badly the domain fails to be convex.
The precise curvature lower bound $S>0$ plays no role whatsoever and $S<+\infty$ is needed for regularity purposes only, essentially allowing the application of Sturm's results \cite{sturm2025gradient} and ensuring the validity of \eqref{eq:WC}\eqref{eq:GE}.
One might be able to cover the case of reasonable Lipschitz domains $\Omega\subset\R^d$ (say, polygonal with controlled aperture, as is common in numerical analysis) by applying our results in either increasing sequences of inner domains $\Omega_k\nearrow \Omega$ or decreasing outer $\Omega_k\searrow\Omega$, with $\Omega_k$ satisfying the necessary assumptions for some $S_k\to+\infty$ as $k\to\infty$.
We leave this for some future work.

\begin{appendices}
\section{Appendix}
\label{sec:appendix}

\begin{lem}[Benamou-Brenier formula \cite{BB,Lisini2007}]\label{lem:BBnon-convex}
    If $(\mu,F)=(\mu,\mu v)$ solves the continuity equation \eqref{eq:CE} with $\int_0^T \int_M |v_t|^2\, \rd\mu_t \rd t<\infty$, then (up to redefining $\mu$ on a negligible set of times) $\mu\in AC_2([0,T];W)$ and
    \begin{equation*}
        |\dot\mu_t|^2 \leq  \int_M |v_t|^2 \,\rd\mu_t
        \hspace{1cm}\text{for a.e. }t\in[0,T].
    \end{equation*}
    Conversely, if $\mu \in AC_2([0,T];W)$, then there exists a (minimal) vector field $v:[0,T]\to TM$ such that $(\mu,\mu v)$ solves the continuity equation \eqref{eq:CE} and
    \begin{equation*}
        |\dot\mu_t|^2 = \int_M |v_t|^2 \rd\mu_t
        \hspace{1cm}\text{for a.e. }t\in[0,T].
    \end{equation*}
    Finally, $v$ is minimal if and only if $v_t \in \overline{\{\nabla\phi,\,\phi\in C^1(M)\}}^{L^2(\mu_t)}$ for a.e. $t\in[0,T]$ (and it is unique).
\end{lem}
\begin{proof}
    The statements for manifolds without boundary can be found in \cite[Thm.~2.29, Eq.~2.23]{AGusers}, see also \cite[Rmk.~8]{Lisini2007}, \cite[Thm.~8.4.5]{AGS}. In our case, we can always see $M$ inside a manifold without boundary of the same dimension, and use the aforementioned results. 
    Indeed, we can see $M$ embedded inside its double $\hat M$ (which is a smooth differential manifold) and then extend the metric of $M$ smoothly on $\hat M$, see \cite[Ex.~6.44]{LeeRM}.
\end{proof}

\begin{lem}[weak-strong convergence]
\label{lem:weak_strong_CV}
Let $X$ be a Polish space and $\mu_n,\mu\in \P(X)$ be Radon measures such that $\mu_n\narrowcv\mu$ as $n\to\infty$.
Take two sequences $u_n,v_n\in L^2(\mu_n)$ converging weakly to $u,v\in L^2(\mu)$ in the sense that
$$
\int_X  u_n \xi\, \rd \mu_n \to \int_X  u \xi\,\rd \mu
\qqtext{and}
\int_X  v_n\xi\, \rd \mu_n \to \int_X  v \xi\,\rd \mu
$$
for all $\xi\in C_c(X)$.
Assume in addition that $u_n$ converges in norm, i.e.
\begin{equation}
\label{eq:convergence_norms}
\int_X  |u_n|^2 \rd \mu_n \to \int_X  |u|^2\rd \mu.
\end{equation}
Then
$$
\int_X u_n v_n\,\rd\mu_n \to \int_X u v\,\rd\mu.
$$
\end{lem}
This a variant on two usual facts that a) ``weak convergence and convergence of the norm implies strong convergence'' for $u_n$, combined with b) the product $u_n v_n$ of strongly with weakly converging sequences converges.
However due to the varying base measure one cannot really make sense of ``$u_n$ converges strongly'', and our convergence in norm \eqref{eq:convergence_norms} is actually the definition of strong convergence in \cite[definition 5.4.3]{AGS}.
To the best of knowledge this is nowhere written in the literature in this form.
\begin{proof}
Take any $U\in C_c(X)$, and write
$$
u_n v_n \,\rd\mu_n -u  v\, \rd\mu
=(u_n-U) v_n \,\rd\mu_n
+  \left(U v_n\,\rd\mu_n - U v\,\rd\mu \right)
+ (U-u) v\,\rd\mu .
$$
Then by Cauchy-Schwarz inequality
\begin{multline*}
 \left|\int_X u_n v_n\,\rd\mu_n - \int_X u v\,\rd\mu \right|
\leq \|u_n-U\|_{L^2(\mu_n)}\|v_n\|_{L^2(\mu_n)}
\\
+\underbrace{\left|\int_X U v_n\,\rd\mu_n - \int_X U v\,\rd\mu \right|}_{\xrightarrow[n\to\infty]{} 0}
+\|U-u\|_{L^2(\mu) }\|v\|_{L^2(\mu)}
\end{multline*}
It is not difficult to check as usual that, being weakly converging, $v_n$ is bounded in $L^2(\mu_n)$.
Convergence in norm \eqref{eq:convergence_norms} gives
\begin{multline*}
\|u_n-U\|_{L^2(\mu_n)}^2
=\int_X |u_n-U|^2\rd\mu_n
=\int_X |u_n|^2\rd\mu_n + \int_X|U|^2\rd\mu_n - 2 \int_X u_n U\rd\mu_n
\\
\xrightarrow[n\to\infty]{}
\int_X |u|^2\rd\mu  + \int_X|U|^2\rd\mu  - 2 \int_X u U\rd\mu
=\|u-U\|^2_{L^2(\mu)},
\end{multline*}
and taking $\limsup$ in the previous inequality yields
$$
\limsup\limits_{n\to\infty}
\left|\int_X u_n v_n\,\rd\mu_n - \int_X u v\,\rd\mu \right|
\leq C\|u-U\|^2_{L^2(\mu)}
$$
with $C=\limsup\|v_n\|_{L^2(\mu_n)}+\|v\|_{L^2(\mu)}$ finite and independent of $U$.
Being $\mu $ a Radon measure, $C_c(X)$ is dense in $L^2(\mu )$.
Choosing $U\in C_c(X)$ appropriately we can finally make $\|u-U\|^2_{L^2(\mu)}$ arbitrarily small and the claim follows.
\end{proof}
\begin{lem}
\label{lem:nablarho_is_gradient} 
Let $\mu=\rho V$ have finite Fisher information $\I(\mu)<+\infty$ with density bounded from above and from below, $0<\lambda^2\leq\rho(x)\leq \Lambda^2<+\infty$.
Then $\nabla\log\rho\in \overline{\{\nabla\phi,\,\phi\in C^1(M)\}}^{L^2(\mu)}$.
\end{lem}
\begin{proof}
By assumption we have that $u=\sqrt{\rho}\in H^1$ with $0<\lambda\leq u(x)\leq \Lambda$.
Let $u^n=P_{1/n}u$.
By standard properties of the heat flow $u^n$ belongs to $C^1(M)$, converges to $u$ in strongly in $H^1$, and satisfies the same bounds $\lambda\leq u^n(x)\leq \Lambda$.
Letting $\rho^n=(u^n)^2$ and $\phi^n=\log(\rho^n)=2\log u^n\in C^1(M)$, a direct computation gives
$$
\int_M\left|\nabla\phi^n-\nabla\log\rho\right|^2\rd\mu
=
\int_M\left|2\frac{\nabla u^n}{u^n}-2\frac{\nabla u}{u}\right|^2 u^2\rd V
=4\int_M\left|\frac{u}{u^n}\nabla u^n-\nabla u\right|^2\rd V.
$$
Since $\nabla u^n\to \nabla u$ and $u^n\to u\geq \lambda>0$ in $L^2$, we have, up to a subsequence, that $f^n=\left|\frac{u}{u^n}\nabla u^n-\nabla u\right|^2\to 0$ a.e.
Moreover, the strong $L^2$ convergence of $\nabla u^n$ implies that $|\nabla u^n|^2$ is uniformly integrable.
Together with the uniform bounds $0\leq \frac{u}{u^n}\leq \frac{\Lambda}{\lambda}$, we see that $|f^n|\leq 2\frac{\Lambda^2}{\lambda^2}|\nabla u^n|^2 + 2 |\nabla u|^2$ is uniformly integrable, hence by Vitali's convergence theorem $f^n\to 0$ in $L^1(V)$ and the statement follows.
\end{proof}

\end{appendices}
\section*{Acknowledgments}
J.-B. C. was funded by FCT - Funda\c{c}\~ao para a Ci\^encia e a Tecnologia through project UID/04561/2025 (DOI 10.54499/UID/04561/2025).
M.F. acknowledges funding by the Deutsche Forschungsgemeinschaft (DFG, German Research Foundation) under Germany’s Excellence Strategy – EXC-2047/1 – 390685813.
L.M. was funded by FCT - Funda\c{c}\~ao para a Ci\^encia e a Tecnologia through a personal grant 2020/00162/CEECIND (DOI 10.54499/2020.00162.CEECIND/CP1595/CT0008) as well as project UID/00208/2025 (DOI 10.54499/UID/00208/2025), and wishes to thank Professor Doctor H. Lavenant for numerous discussions on the regularization of continuity equations.

\bibliographystyle{plain}
\bibliography{biblio}
\end{document}